\documentclass[reqno]{amsart}

\let\tmp\oddsidemargin
\let\oddsidemargin\evensidemargin
\let\evensidemargin\tmp
\reversemarginpar

\usepackage{float}
\usepackage{todonotes}
\usepackage{bbm}
\usepackage{bm}

\usepackage[utf8]{inputenc} 

\usepackage{amsthm}

\usepackage{imakeidx} 
\usepackage{verbatim}
\usepackage{graphicx} 
\usepackage{enumerate} 
\usepackage{hyperref}
\hypersetup{
    colorlinks,
    filecolor=black,
    linkcolor=blue,
    urlcolor=black}

\newtheorem{theorem}{Theorem}[section]
\newtheorem{proposition}[theorem]{Proposition}
\newtheorem{lemma}[theorem]{Lemma}

\newtheorem{corollary}[theorem]{Corollary} 

\theoremstyle{definition}

\newtheorem{definition}[theorem]{Definition}

\newtheorem{remarkx}[theorem]{Remark}
	{\popQED\\}{\endremarkx}

\newcommand\Om{\Omega}

\newcommand\ObsDynOp{A}

\newcommand\ObsDynSp{\mathcal{H}} 

\def\bdem{\begin{proof}}
\def\edem{\end{proof}}
\def\bequ{\begin{equation}}
\def\eequ{\end{equation}}

\newcommand{\C}{\mathbb{C}}

\definecolor{red}{RGB}{255,0,0}

\newcommand{\calB}{\mathcal B}

\newcommand{\calH}{\mathcal H}

\newcommand{\calS}{\mathcal S}         
\newcommand{\calT}{\mathcal T}

\newcommand{\HH}{\mathcal{H}}

\newcommand{\R}{\mathbb{R}}

\newcommand{\N}{\mathbb{N}}
\newcommand{\CC}{\mathbb{C}}

\newcommand{\G}{\mathcal{G}}
\newcommand {\D} {\mathbb D}

\usepackage{makeidx}
\makeindex

\usepackage{amsmath,amssymb}

\title{Continuous and discrete dynamical sampling}
\author{Roc\'\i o D\'\i az Mart\'\i n, Ivan Medri, Ursula Molter}
\date{}

\begin{document}

\maketitle

\begin{abstract} In this paper we study the continuous dynamical sampling problem at infinite time in a complex Hilbert space $\mathcal{H}$. We find necessary and sufficient conditions on a bounded linear  operator $A\in\mathcal{B}(\mathcal{H})$ and a set of vectors $\G\subset \mathcal{H}$, in order to obtain that $\{e^{tA}g\}_{g\in\G, t\in[0,\infty)}$ is a semi-continuous frame for $\mathcal{H}$. We study if it is possible to discretize the time variable $t$ and still have a frame for $\mathcal{H}$. We also relate the continuous iteration $e^{tA}$ on a set $\G$ to the discrete iteration $(A^\prime)^n$ on $\G^\prime$ for an adequate operator $A^\prime$ and set $\G^\prime\subset \mathcal{H}$.   
\end{abstract}


\section{Introduction}

\subsection{The dynamical sampling problem: discrete and continuous time}
\
\
The classical sampling and reconstruction problem consists in recovering a function $f:\calH\rightarrow \C$ from the knowledge of its values at certain points of the spatial domain $\calH$. In the dynamical sampling problem, the set of \textit{space samples} is replaced by a set of \textit{space-time samples}.

In order to state the general dynamical sampling problem, let $f$ be a function in a complex separable Hilbert space $(\mathcal{H}, \langle\cdot,\cdot\rangle)$. Assume that $f$ evolves through an evolution operator $A: \mathcal H \to \mathcal H$ so that the function at time $t$ has evolved to become $f^{(t)}=A^{(t)}f$. The concrete definition of $A^{(t)}$ will depend on whether the evolution happens in discrete or continuous time and will be introduced later. 
Let $\G=\{g^j\}_{j\in J}$ be a set of vectors in  $\HH$,  where $J$ is a countable set (finite or infinite) of indexes.
Consider a set $\calT\subseteq[0, =\infty) $ which can be of four types: a discrete finite set, $\calT=\{0,1,\dots,N\}$ (for $N\in\N\cup\{0\}$) or a discrete infinite  set $\calT=\N\cup \{0\}$, or a finite interval, $\calT=[0,L]$ (where ($L\in[0,\infty)$) or $\calT=[0,\infty)$.
The time-space sample at time $t\in \calT$ and location $j\in J$, is the value 
$$f^{(t)}(j):=\langle A^{(t)}f,g^j\rangle.$$ 
In this way we associate to each pair $(j,t)\in J\times \calT$
 a sample value.  
The general dynamical sampling problem can then be described as: \textit{Under what conditions on the operator $A$, the vectors $g^j$ and  the set $J\times \calT$, can  every vector $f$ in the Hilbert space $\calH$ be recovered in a stable way from the samples }
\begin{equation}\label{set_samples}
 \{f^{(t)}(j): \ (j,t) \in J\times \calT\}.    
\end{equation}
Here {\em stable way} means, as usual, that the reconstruction is robust under small perturbations. We will make this notion precise in a moment.

Throughout the paper, we assume that $A\in\calB(\HH)$.  
 For each $j\in J$, let $S_j$ be the operator 
\begin{equation*}
    S_j: \overline{\text{span}}\{g^j: \, j\in J\}\rightarrow L^2(\calT,\mu_{\calT}) \qquad
    (S_j f) (t):= f^{(t)}(j),
\end{equation*}
where $\mu_\calT$ is either the discrete measure or the Lebesgue measure depending on the set $\calT$. Define $\calS$ to be the operator $\calS:=\bigoplus_{j\in J}S_j$. 
We say that $f$ can be recovered from \eqref{set_samples} in a stable way if  $\calS$ is a bounded and invertible (on its range) linear operator. That is, if and only if there exist constants $c,C>0$ such that for all $f\in\mathcal{H}$
\begin{equation}\label{frame_cond}
c\|f\|^2_2\leq\|\mathcal S f\|^2_2=\sum\limits_{j\in J} \|S_jf\|_{L^2(\mu_\calT)}^2= \sum\limits_{j \in J} \int_{\calT}|\langle A^{(t)}f, g^j\rangle|^2 d\mu_\calT(t)\leq C\|f\|^2_2.
\end{equation}

The choices for $A^{(t)}, t \in \calT$ will be as follows: The  \textit{discrete dynamical sampling problem} is the case in which the iterations of the operator $A$ are discrete   and we take $A^{(n)}:=A^n$ for $n\in \calT$. Instead, for the \textit{continuous dynamical sampling problem} we  take $A^{(t)}:=e^{tA}$ for  $t\in \calT$, where $\calT=[0,L]$ or $\calT=[0,\infty)$.

The problem of dynamical sampling for discrete time was the first to be formulated. It  was deeply studied for example in \cite{ItNormOp,Tang, aaa,a1,DSOFIS, suarez}. In particular in \cite{Tang, DSOFIS, suarez} necessary and sufficient conditions on a normal operator $A$ and  $\G\subset \mathcal{H}$ are obtained, so that $\{A^n  g\}_{g\in \G, n\in\N}$ is a frame for $\mathcal{H}$ when $|\G|<\infty$.  
        
The problem of dynamical sampling for continuous time was originally posed for normal operators in \cite{FIACPO}. There, since $A$ is normal, the authors work with continuous powers $\{A^t\}_{t\in[0,L]}$  defined though the Spectral Theorem. When $A$ is a self-adjoint and strictly positive operator, the operators $A^t$ coincide with $e^{t\widetilde{A}}$ for an appropriate choice of  $\widetilde{A}\in\calB(X)$, for all $t\geq 0$. When $A$ is not self-adjoint or not strictly positive, it is not clear how to find  (or decide if there exists) an appropriate $\widetilde{A}$ that satifies the equality (see Definition~\ref{normal-eat}).  

Since $e^{tA}=\sum_{n=0}^\infty \frac{t^n}{n!} A^n$ is always defined for bounded operators (even for non normal ones, see \eqref{semi}), in this article we will work with $e^{tA}$ (instead of $A^t$). This choice is also justified by the role of  $e^{tA}$  in continuous dynamical systems $\dot{x}(t)=Ax(t)$ as shown in \cite{dmm}. 
Hence we will not require (a priori) that $A$ is normal. Further, since $e^{tA}$ is always invertible, the invertibility of $A$ will not be necessary.  Note that, as pointed out before, in  general, it is not true that $e^{tA} = \tilde A^t$,  and hence the results from \cite{FIACPO} cannot be obtained immediately, although most of them will still be valid. 

For the discrete dynamical problem  ($\calT=\{0,1,...,N\}$ or $\calT=\N\cup\{0\}$),  condition  \eqref{frame_cond} becomes
\begin{equation*}
c\|f\|^2_2\leq \sum_{g\in\G} \sum_{n\in\calT}|\langle f, (A^*)^n g\rangle|^2 \leq C\|f\|^2_2, \qquad (\forall f\in\mathcal{H}),
\end{equation*}
where $A^*$ is the adjoint operator of $A$ (see for example  \cite[Lemma 1.2]{Tang}), which is exactly the condition for the set $\{(A^*)^n g\}_{g\in\G, n\in \calT}$ in to. be a {\em frame} in $\HH$.
For the continuous dynamical problem, whith $\calT=[0,L]$ or $\calT=[0,\infty)$, condition \eqref{frame_cond} can be written as
\begin{equation*}
c\|f\|^2_2\leq\sum\limits_{g\in \G} \int_{\calT}|\langle f, e^{tA^*}g\rangle|^2 d\mu_\calT(t)\leq C\|f\|^2_2 \qquad (\forall f\in\mathcal{H}).    
\end{equation*}
This condition says that  $\{ e^{tA^*}g\}_{g\in \G, t\in \calT}$ is a \textit{semi-continuous} frame for $\HH$ (see for example \cite{FS18}). For simplicity of notation we will work with $\{e^{tA}g\}$ and $\{A^ng\}$ instead of $\{e^{tA^*}g\}$ and $\{(A^*)^ng\}$.

\subsection{Motivations, Contributions and Organization}\

The  goal of this paper is to find necessary and sufficient conditions for the continuous time dynamical sampling problem to be solvable with only a finite number of spatial samples, that is, $|\G|<\infty$. Since we work on an infinite dimensional Hilbert space, the only possibility then is to have an infinite number of time samples. When the time parameter is continuous (either $[0,\infty)$ or $[0,L]$), we always have infinite time samples. However, by the results in \cite{FIACPO}, one can show that  for $t \in [0,L]$ in a finite interval, $\{e^{tA}g\}_{g\in\G, t\in[0,L]}$ is a frame if and only if  $\{e^{tA}g\}_{g\in\G,t\in T}$ is a frame, with $T$ a finite set of points. Therefore, when considering a finite number of space samples, the only possibility to obtain a solution will be to sample considering an infinite interval of time, that is $\mathcal{T}=[0,\infty)$. 

However, for operators giving rise to exponentially stable semigroups, it can be shown  (see Corollary \ref{exp stab implies G inf}), that even in this case, it will never suffice to take finite space samples. Instead, for operators that generate non-exponentially stable semigroups we are able to obtain positive results.  
Our approach here is to relate the continuous problem to the discrete one.  By establishing the sought after  relations, we are able to translate the results from the discrete time dynamical sampling problem to the continuous one.

Even though the main objective is to reduce the spatial samples to a finite number, as a consequence of our results  we are also able to reduce the amount of samples but in the time variable. We obtain this through a discretization procedure for all operators; giving rise to exponentially stable or non-exponentially stable semigroups. In fact,  for the exponentially stable case we show that sampling up to infinite in time is equivalent to only sampling at a finite set of time values (see Section \ref{discretization}).

Finally, we were able to establish the precise connection between the discrete dynamical sampling problem and the continuous one, obtaining Theorems \ref{teo_bon1} and Theorem \ref{Cont disc normal op}. We do so by first relating both problems and then adapt results from one to the other.

The paper is organized as follows: In Section \ref{Preliminaries} we set the background needed to state and prove our results. %
In Section \ref{discretization} we show how to   discretize the sampling process without altering the dynamics (Theorems \ref{ScToDscr} and \ref{ScToDscrInfiniteTime}, and Corollaries \ref{exponentially stable if only if infinite implies finite} and \ref{exp stab implies G inf}). 
In Section \ref{diagonalization}, we show the connection between the discrete dynamical sampling problem and the continuous one via a transformation in the dynamics (see Theorems \ref{teo_bon}, \ref{teo_bon1} and \ref{Cont disc normal op}). As a consequence we then can translate the results from the discrete time dynamical sampling problem to the continuous one. This is done in Section \ref{sec_suarez}. Necessary and sufficient conditions for $\{e^{tA}g\}_{g\in\G, t\in[0,\infty)}$ to be a semi-continuous frame are the content of Theorem \ref{suarez_cont}.  
Some of the results of this paper were announced without proofs in \cite{dmm}.

\section{Preliminaries}\label{Preliminaries}

\subsection{Reproducing Kernel Hilbert Spaces and Hardy Spaces }\

\begin{definition} A nonzero Hilbert space $\HH$ of analytic functions on a plane domain $\Omega\subset\C$ is said to be a  reproducing kernel Hilbert space (RKHS) if the point evaluation is a continuous linear functional on $\HH$. By the Riesz Representation Theorem, in this situation, for every $w\in \Omega$, there exists a function ({\em reproducing Kernel}\,) $z\mapsto k_w(z)$ belonging to $\HH$, such that, for every $f\in \HH$,
\begin{equation}
 f(w)=\langle f,k_w\rangle.   
\end{equation}
Note that in particular, $\|k_w\|^2=k_w(w)$, $\forall w\in \Omega$.
\end{definition}

\begin{definition}
    A sequence of vectors $\{\varphi_n\}$ in a Hilbert space $\mathcal{H}$ is said to be
        a \textit{Riesz sequence} if there exist positive constants $c$ and $C$ such that for every finite sequence $(a_n)$ 
        \begin{equation*}\label{rdm-riesz_cond}
            c\sum |a_n|^2\leq \|\sum a_n \varphi_n\|^2 \leq C\sum |a_n|^2.
        \end{equation*}
    If in addition  $\{\varphi_n\}$ is a complete system in $\mathcal{H}$, it is called
      a \textit{Riesz basis}. 
      \end{definition}

\begin{definition}\label{def_interpolating} Let $\HH$ be a RKHS of analytic functions on a plane domain $\Omega\subset\C$, with reproducing kernel $w\mapsto k_w$ (for $w\in \Omega$). Let $\Lambda$ be a discrete subset of complex numbers in $\Omega$.
Let $\{f_\lambda\}_{\lambda\in\Lambda}\subset \HH$ the subset of normalized reproducing kernels defined by
\begin{equation*}
    f_\lambda(z):=\frac{k_\lambda(z)}{k_\lambda(\lambda)}  \qquad (z\in\Omega).     
\end{equation*}
Denote by
\begin{equation}
    \ell^2_{\Lambda}:=\left\{ (c_{\lambda})_{\lambda\in\Lambda} : \, \sum_{\lambda\in\Lambda}\frac{|c_\lambda|^2}{k_\lambda(\lambda)}<\infty \right\}.
\end{equation}
\noindent       $\Lambda$ is said to be a \textit{(complete) interpolating sequence} for $\HH$ if $\{f_\lambda\}_{\lambda\in\Lambda}$ is a Riesz sequence (basis) in $\HH$. That is, for every sequence $(c_\lambda)\in \ell^2_{\Lambda}$ there a (unique) $f\in \HH$ such that $f(\lambda)=c_\lambda$ for all $\lambda\in\Lambda$.
\end{definition}

Hardy spaces are a particular example of RKHS. We recall the notions of Hardy spaces to set the notation. As a reference see for example \cite{Hoffman Hardy}.

\subsubsection{%
{Hardy space of the right half plane}}

\begin{definition}
  Let  
      \begin{equation*}
\C_+:=\left\{z = x+iy \in \C: \, x=\mathop{\mathrm{Re}}(z)>0\right\},       
\end{equation*}
be the right complex half plane.  The Hardy space $H^2(\C_+)$ can
be defined as the set of all analytic functions $f:\C_+\rightarrow \C$ such that
    \[\|f\|_{H^2(\C_+)} := \left(\sup_{x>0}\int_{-\infty}^\infty |f(x+iy)|^2 dy\right)^{1/2} < \infty. \] 
\end{definition}

If $f\in H^2(\C_+)$, its boundary values $\widetilde{f}(iy) := \lim_{x\rightarrow 0^+} f(x+iy)$ , are defined  almost everywhere, and the boundary function $\widetilde{f}$ lies in $L^2(i\R)$ and satisfies $\|\widetilde{f}\|_{L^2(i\R)}= \|f\|_{H^2(\C_+)}$. We  identify $f$ and $\widetilde{f}$, and thus $H^2(\C_+)$ can naturally be seen as a closed subspace of $L^2(i\R)$ and hence a Hilbert space.

With the  notation above we can define the inner product on $H^2(\C_{+})$  by
\begin{equation*}\label{innerprod}
    \langle f, g\rangle_{H^2(\C_+)}:=\int_{-\infty}^\infty \widetilde{f}(iy)\overline{\widetilde{g}(iy)} \ dy,
\end{equation*}
which makes $H^2(\C_+)$ a RKHS, with reproducing kernel
\begin{equation*}
 k^{\C_+}_s(z):=\frac{1}{2\pi(z+\overline{s})}  \quad (s\in \C_+) \qquad \text{with norm} \qquad  \|k_s\|_{H^2(\C_+)}^2= \frac{1}{4\pi\mathop{\mathrm{Re}}(s)}.
\end{equation*}

\subsubsection{
{Hardy space of the disc}}

\begin{definition}
    The Hardy space $H^2(\D)$ of the disc $\D$, given by  
    \begin{equation*}
        \D:=\left\{z\in \C: |z|<1 \right\},       
    \end{equation*} 
   is defined as the set of all analytic functions $f:\D\rightarrow \C$ such that
    \[\|f\|_{H^2(\D)} := \left(\sup_{r<1}\frac{1}{2\pi}\int_{0}^{2\pi} |f(re^{i\omega})|^2 d\omega\right)^{1/2} < \infty. \]
\end{definition}

Again, if $f\in H^2(\D)$, its boundary values $\widetilde{f}(e^{i\omega}) := \lim_{r\rightarrow 1} f(re^{i\omega})$ exist almost everywhere, and the boundary function $\widetilde{f}$ lies in $L^2(\mathbb{T})$ and satisfies $\|\widetilde{f}\|_{L^2(\mathbb{T})}= \|f\|_{H^2(\D)}$. We  identify $f$ and $\widetilde{f}$, and thus consider $H^2(\D)$ as a closed subspace of $L^2(\mathbb{T})$, with
 inner product 
\begin{equation*}\label{innerproddisc}
    \langle f, g\rangle_{H^2(\D)}:=\frac{1}{2\pi}\int_{0}^{2 \pi} \widetilde{f}(e^{i\omega})\overline{\widetilde{g}(e^{i\omega})} \ d\omega,
\end{equation*}
which as before makes  $H^2(\D)$ a RKHS with kernel, 
\begin{equation*}
k^{\mathbb{D}}_s(z)=\frac{1}{1-z\overline{s}}  \quad (s\in \D) \qquad \text{with norm} \qquad  \|k_s\|_{H^2(\D)}^2= \frac{1}{1-|s|^2}.
\end{equation*}

\subsubsection{
{Relations between Hardy Spaces.}}\ 

There is a natural isometric isomorphism between Hardy spaces on the half-plane $\C_+$ and on the disc $\D$ induced  by the self-inverse bijection given by 
\begin{equation}\label{M del D en C}
\mathfrak{h}:\mathbb{D}\to\C_+ \qquad \mathfrak{h}(z):=\frac{1-z}{1+z}.    
\end{equation}
It is easy to see that $\mathfrak{h}\circ \mathfrak{h}(z)=z$,  
and that for $f\in H^2(\C_+)$,  $g(z):=f\left(\frac{1+z}{1-z}\right)$ belongs to $H^2(\mathbb{D})$.
\begin{theorem}\label{iso_hardys}
The mapping $V:H^2(\mathbb{D})\to H^2(\C_+)$ defined by
\begin{equation}\label{iso de Hardy}
    (Vf)(s):=\frac{1}{\sqrt{\pi}(1+s)}f(\mathfrak{h}(s)),
\end{equation}
is an isometric isomorphism.
\end{theorem}

The isometric isomorphism $V$ defined in \eqref{iso de Hardy} does not map the reproducing kernel of $H^2(\mathbb{D})$ to the reproducing kernel of $H^2(\C_+)$ but we have
\begin{align}
\text{for } s\in \D, \qquad   V(k_{s}^\mathbb{D})(z) & =  \frac{2\sqrt{\pi}}  {1+\overline{s}}\ k_{\mathfrak{h}(s)}^{\C_{+}}(z)
\notag\\
\text{and for } s\in \mathbb{C}_+, \qquad   V^{-1}(k_{s}^{\mathbb{C}_+})(z) & = \frac{1}  {\sqrt{\pi}(1+\overline{s})}\ k_{\mathfrak{h}(s)}^{\D}(z)\label{k en k 2}.
\end{align}

As an easy consequence of the isomorphism between the Hardy spaces we have the following proposition.
\begin{proposition}
     $\{\lambda_j\}_{j=1}^\infty\subset \C_+$ is an interpolating sequence for $H^2(\C_+)$ if and only if $\{\mathfrak{h}(\lambda_j)\}_{j=1}^\infty$ is an interpolating sequence for $H^2(\mathbb{D})$, where $\mathfrak{h}$ is as in \eqref{M del D en C}. 
\end{proposition}

\subsection{Previous results about the discrete dynamical sampling problem}

Let us recall here some results of the discrete dynamical sampling problem.  
For the convenience of the reader and to make this paper self-contained,  we repeat (without proofs) the relevant results from the recent article
\cite{suarez}.

 Write $\varphi_0(z):=z$ and for $0\neq\eta_j\in \D$, $\varphi_{\eta_i}(z) :=   \frac{\eta_i-z}{1-\overline{\eta}_i z}$. 
The pseudo-hyperbolic metric in $\D$ is given by $\rho(z,w):= |\varphi_z(w)|$, and
we denote by $\Delta_\rho(z,r) := \{ w\in \D: \ \rho(z,w) <r \}$ for $0<r<1$.
\begin{definition}
A positive measure $\mu$ on $\D$ is called a {\em Carleson measure} if there exists a constant $C>0$ such that
$$\int |f|^2 d\mu \le C^2 \|f\|^2, \qquad \forall \  f\in H^2(\D).
$$
\end{definition}

We have the following Theorem.
\begin{theorem}\cite{suarez}\label{teo_imp_suarez}
Let $A$ be a diagonal operator with respect to the standard basis
in $\ell^2(J)$, where $J$ is finite, or $J= \mathbb{N}$, and let $a^1, \ldots , a^m \in \ell^2(J)$.
Then $\{ A^n a^j: \, n\in \mathbb{N}\cup \{0\}, \, 1\le j\le m\}$ is a frame if and only if
\begin{itemize}
\item There exists $C\geq1$ such that each $a^i$ ($i=1, \ldots , m$) is given by 
\begin{equation}\label{niunii}
    a_j^i =  d_j  \, \overline{\alpha}^i_j (1-|\eta_{j}|^{2})^{\frac{1}{2}} \qquad  (\forall j\in J), \quad \text{ with } \quad \sum_{i=1}^m|\alpha^i_j|^2 =1\, \text{ for all } j \in J.
\end{equation}
\item The sequence of eigenvalues $\{ \eta_j : \, j\in J\}$ of $A$ is in $\D$ and satisfies that the measure $\sum_{i}  (1-|\eta_i|^2) \delta_{\eta_i}$ is a Carleson measure. 
\item The sequence $S:=\{ \eta_j \}$ and the double sequence
$\{ \alpha^i_j : \, j\in J, \,1\le i\le m\}$ of \eqref{niunii} satisfy the following two conditions:
\begin{enumerate}
\item[(1)] there is $\beta>0$ such that $\Delta(\eta_j , \beta)$ contains no more than m points of $S$
counting repetitions for all $j$.
\item[(2)] there is $0<\gamma <\beta$ such that
if\/ $\eta_{j_1} , \ldots, \eta_{j_p}$ ($p\leq m$) are the points of $S$ in
$\Delta(\eta_{j_1}, \gamma)$  counting repetitions, the related matrix satisfies for all $(c_{j_1},\ldots, c_{j_p}) \in \mathbb{C}^p$,
$$
D
\left\|
    \begin{bmatrix}
    c_{j_1} \\
    \vdots  \\
    c_{j_p} \\
  \end{bmatrix}
\right\|^2_{\mathbb{C}^p}
\le
\left\|
  \begin{bmatrix}
    \alpha^1_{j_1}  & ...  &\alpha^1_{j_p} \\
    \vdots     &            &  \vdots  \\
     \alpha^m_{j_1}  & ...  &\alpha^m_{j_p}\\
  \end{bmatrix}
  \begin{bmatrix}
    c_{j_1} \\
    \vdots \\
    c_{j_p} \\
  \end{bmatrix}
\right\|^2_{\C^{m \times 1}}
$$
where $D>0$ does not depend on $p$ or the $\alpha$'s .
\end{enumerate}
\end{itemize}
\end{theorem}

Note that, in the particular case, when $m=1$ choosing $a^1_j=(1-|\eta_{j}|^{2})^{\frac{1}{2}}$, the orbit $\{ A^n a^1\}_{ n\in\N\cup\{0\}}$ is a frame if and only if
$\left\{\frac{k^\D_{\eta_j}}{\|k^\D_{\eta_j}\|} \right\}$ is a Riesz sequence which in turn is equivalent to $\{ \eta_j\}$ being an interpolating sequence. Therefore the result in \cite[Theorem 3.14]{Tang} is included.

\subsection{Exponential map and Spectral theorem}\label{prelim_exp}\

We list now from \cite{E, Pazy} some results (without proofs) on semigroup theory for the case of bounded operators that we will need later.

An operator $A\in \mathcal{B}(\mathcal{H})$ defines an one-parameter group
\begin{equation}\label{semi}
  \left \{ e^{tA}:=\sum_{n=0}^\infty \frac{t^n}{n!}A^n, \qquad t\in\R \right\}.
\end{equation}

We have the following results.
\begin{proposition}\  
\begin{enumerate}
    \item Given $f\in\mathcal{H}$, $e^{tA}f$ is the solution of the differential equation $\dot{x}(t)=Ax(t)$, $x(0)=f$.
    \item $e^{0A} = I$ (the identity operator), $e^{(t+s)A} = e^{tA}e^{sA}$ for all $t,s\in\R$.
    \item $\lim_{t\to s}\|e^{tA}-e^{sA}\|=0$. That is, $e^{tA}$ is a continuous in the operator norm.
    \item For each $t\in\R$, $e^{tA}\in\mathcal{B}(\mathcal{H})$ is an invertible operator.
    \item There exists constants $M\geq 1$ and $\omega$ and  such that
    \begin{equation*}
    \|e^{tA}\|\leq M e^{\omega t} \qquad \forall t\geq 0.
    \end{equation*}
    In particular, one can choose $M=1$ and $\omega=\|A\|$. 
\end{enumerate}
\end{proposition}

\begin{definition}\label{def_exp_estable} Let $A \in \calB(\HH)$ be such that $\|e^{tA}\|\leq M e^{\omega t}\ \forall t\geq 0$. If $\omega<0$ we say that the semigroup $\{e^{tA}\}_{t\in[0,\infty)}$ is \textit{exponentially stable}. Since $A$ is a bounded operator, this is equivalent to the condition $\mathop{\mathrm{Re}}(\lambda) < 0$ for all $\lambda \in \sigma(A)$, where $\sigma(A)$ denotes the spectrum of $A$.    
\end{definition}

For some of the results it will be convenient to use a spectral definition of the exponential map $e^{tA}$. Therefore we recall the spectral theorem with multiplicity for normal operators \cite{conway} which will yield an alternative (but equivalent) definition of the exponential map.

For a  non-negative regular Borel measure $\mu$ on $\mathbb{C}$, $N_{\mu}$  will denote the multiplication operator acting on  $ L^2(\mu)$, i.e., for a $\mu$-measurable function $f:\CC\to \CC$ such that $\int_\CC |f(z)|^2d\mu(z)< \infty$, 
$$N_{\mu}f(z)=zf(z).$$

We will use the notation $[\mu]=[\nu]$ to denote two mutually absolutely continuous measures $\mu$ and $\nu$.

The operator $N_{\mu}^{(k)}$ will denote the direct sum of $k$ copies of $N_{\mu}$, i.e.,
\begin{equation*}
(N_\mu)^{(k)}=\oplus_{i=1}^{k}N_{\mu}.
\end{equation*}
Similarly, the space $(L^2(\mu))^{(k)}$ will denote the direct sum of $k$ copies of $L^2(\mu)$. 

\begin{theorem}[\textbf{Spectral theorem with multiplicity}]\label{spectral theorem} For any normal operator $A$ on $\HH$ there are mutually singular  non-negative  Borel measures $\mu_j, 1\leq j\leq\infty$, such that $A$ is equivalent to the operator
	$$N_{\mu_{\infty}}^{(\infty)}\oplus N_{\mu_1}
	\oplus N_{\mu_2}^{(2)}\oplus\ldots,$$
	i.e., there exists a unitary transformation
	\begin{equation*}\label{calW}
	U:\HH\rightarrow \mathcal{W}:=(L^2(\mu_{\infty}))^{(\infty)}\oplus L^2(\mu_1)\oplus (L^2(\mu_2))^{(2)}\oplus\ldots    
	\end{equation*}
	such that
	\begin{equation*}\label{representation of normal}
	UAU^{-1}=N_{\mu_{\infty}}^{(\infty)}\oplus N_{\mu_1}\oplus N_{\mu_2}^{(2)}\oplus\ldots.
	\end{equation*}
	Moreover, if $\tilde{A}$ is another normal operator with corresponding measures $\nu_{\infty},\nu_1,\nu_2,\ldots$, then $\tilde{A}$ is unitary equivalent to A if and only if $[\nu_j]=[\mu_j]$ for $j=1,\ldots,\infty.$
\end{theorem}
A proof of the theorem can be found in \cite{conway}.

Let $U$ be as in Theorem \ref{spectral theorem}. 
 If $g\in \HH$,  one has   $Ug = ((Ug)_j)_{j\in \N^*}$, where $\N^*:=\N\cup\{\infty\}$ $(Ug)_j$ is the restriction of $Ug$ to $(L^2(\mu_j))^{(j)}$. Thus, for any $j\in \N^*$, $ (Ug)_j $ is a function from $ \CC $ to $ \ell^2(\Omega_j) $ and
$$\sum_{j\in \N^*} \quad \int_\CC \|(Ug)_j(z)\|_{\ell^2(\Om_j)}^2d\mu_j(z) =\|g\|^2<\infty .$$

\begin{definition}\label{normal-eat}
	Given a normal operator $A$, $e^{tA}:\HH\rightarrow\HH$ can be defined  by
	\begin{equation*}
	\langle e^{tA} f, g\rangle=\int_{z\in\sigma(A)}e^{tz}\langle Uf(z),Ug(z)\rangle d\mu(z), \quad \text{ for all }f,g\in\HH.    
	\end{equation*}
\end{definition}
 Note that, for  normal operators $A$ this definition coincides with $\sum_{n=0}^\infty \frac{t^n}{n!}A^n$.

\section{Finite vs infinite time and discretization}\label{discretization}

We start with two theorems concerning the discretization in finite ($\mathcal{T}=[0,L]$) and infinite time ($\mathcal{T}=[0,\infty)$). Notice that the first one is for general linear bounded operators $A$, whereas the second is only when $\{e^{tA}\}_{t\in [0,\infty)}$ is exponentially stable.  The proofs are almost verbatim  to \cite[Theorem 5.4]{FIACPO}. However, since we use $e^{tA}$  instead of $A^t$ (as they do in the aforementioned paper), some parts get simplified using some basic results from Section \ref{prelim_exp}. 
We give the proof of Theorem \ref{ScToDscrInfiniteTime} to show where the choice of $e^{tA}$ simplifies the arguments. Following this steps, the proof of Theorem \ref{ScToDscr} can be deduced so we omit it.
Also, the Corollary \ref{frame tiempo finito implica G inf} is the analogous result to the one given in  \cite[Corollary 5.3]{FIACPO}.

\begin{theorem}\label{ScToDscr}	Let $\mathcal{H}$ be a Hilbert space,
	$\ObsDynOp\in\mathcal{B}(\ObsDynSp)$ and let $\G$ be a Bessel system of vectors in $\ObsDynSp$. Then, the following are equivalent.
	\begin{enumerate}
	    \item[i)] $\{e^{tA}g \}_{g\in\G,t\in[0,L]}$ is a semi-continuous frame for $\mathcal{H}$.
	    \item[ii)] There exists $\delta>0$ such that \textbf{for any} finite set $T=\{t_j: j=1,\ldots,n\}$ with $0=t_1< t_2<\ldots<t_n\leq t_{n+1}=L$ and $|t_{j+1}-t_{j}|<\delta$ for all $j\in\{1,...,n\}$,  the system $\{e^{tA}g\}_{g\in \G,t\in T}$ is a frame for $\ObsDynSp$.
	    \item[iii)] There \textbf{exists a} finite partition $T=\{t_j: j=1,\ldots,n\}$ of $[0,L]$ with $0= t_1< t_2<\ldots<t_n\leq L$  such that $\{e^{tA}g \}_{g\in\G,t\in T}$ is a frame for $\ObsDynSp$.
	\end{enumerate}
\end{theorem}

\begin{corollary}\label{frame tiempo finito implica G inf}
    If $\{e^{tA}g\}_{g\in\G, t\in [0,L]}$ is a frame for an infinite dimensional  Hilbert space $\HH$ , then $|\G|=\infty$.
\end{corollary}

If $\{e^{tA}\}_{t\in[0,\infty)}$ is an exponentially stable semigroup, the extension of Theorem \ref{ScToDscr} from $[0,L]$ to $[0,\infty)$ can be done in two different ways: adapting the ideas of Theorem \ref{ScToDscr} or restricting $t$ to  $[0,L]$ and then using Theorem \ref{ScToDscr}. We separate these two approaches in the following two subsections.   

\subsection{Discretization of the infinite time frame: Method 1.}

\begin{definition} We say that a discrete subset of complex numbers $\Lambda$ is  uniformly separated if  $$\inf_{\lambda,\lambda^\prime\in\Lambda, \lambda\not=\lambda^\prime} |\lambda-\lambda^\prime|>0.$$
\end{definition}

\begin{theorem}\label{ScToDscrInfiniteTime}
	Let $\ObsDynOp\in\mathcal{B}(\ObsDynSp)$  and let $\G$ be a Bessel system of vectors in $\ObsDynSp$. If $\{e^{tA}\}_{t\in[0,\infty)}$ is exponentially stable, then the following are equivalent.
	\begin{enumerate}
	    \item[i)] $\{e^{tA}g \}_{g\in\G,t\in[0,\infty)}$ is a semi-continuous frame for $\mathcal{H}$.
	    \item[ii)] There exists $\delta>0$ such that for any uniformly separated countable set of the form $T=\{t_j: j\in \N\}$ with $t_1=0$ and $0<t_{j+1}-t_j<\delta$ $\forall i\in \N$, the system $\{e^{tA}g\}_{g\in \G,t\in T}$ is a frame for $\ObsDynSp$. 
	    \item[iii)] There exists $m>0$ such that for the  countable set $T=\{t_j:= i\cdot m : j=0,1,\ldots \}$,  the system $\{e^{tA}g\}_{g\in \G,t\in T}$ is a frame for $\ObsDynSp$.
	    
	\end{enumerate}
\end{theorem}

\begin{proof} 
    $i)\Rightarrow ii)$ From the assumption that $\G$ is a Bessel sequence in $\ObsDynSp$,  there exists $K>0$ such that $\sum_{g\in\G}|\langle f,g\rangle|^2\leq K\|f\|^2,$ for all $f\in\ObsDynSp$. Since $A$ is a bounded operator and $\{e^{tA}\}_{t\geq 0}$ is exponentially stable,  there are constants $\omega<0$ and $M\geq 1$ such that for any  $0\leq t<\infty$, one has
	\begin{align}\label{AtgBessel2}
	    \sum_{g\in\G}|\langle f,e^{tA}g\rangle|^2&=\sum_{g\in\G}|\langle e^{tA^*}f,g\rangle|^2 \leq KM^2e^{2\omega t}\|f\|^2.	\notag
	\end{align}
	Let $T=\{t_j:j\in\N\}$ be a uniformly separated subset of $[0,\infty)$. Let
    \begin{equation*}
     0<\delta_0:=\inf_{j\in \N}|t_{j+1}-t_j|.   
    \end{equation*}
     For all $j\in\N$ we have that $j\cdot \delta_0\leq t_j$.
    Then,  since $\omega<0$, we obtain
    \begin{equation*}
	\sum_{j\in \N}\sum_{g\in\G}|\langle f,e^{t_jA}g\rangle|^2\leq
	KM^2\left(\sum_{j\in \N}e^{2\omega t_j}\right)\|f\|^2\leq \widetilde{K}\|f\|^2
	\end{equation*} 
	for  $0<\widetilde{K}:=\frac{KM^2}{1-e^{2\omega\delta_0}}<\infty$.
    
    The goal now is to find $\delta>0$ such that for any uniformly separated set $T=\{t_j:j\in\N\}$ with $t_1=0$,  $t_j<t_{j+1}$ for all $j\in \N$ and $t_{j+1}-t_{j}<\delta$,  the system $\{e^{tA}g\}_{g\in \G,t\in T}$ is a frame for $\ObsDynSp$, as long as  $\{e^{tA}g \}_{g\in\G,t\in[0,\infty)}$ is a semi-continuous frame for $\mathcal{H}$, i.e.,  
    \begin{equation*}\label{equation 24 2}
        c\|f\|^2\leq\sum_{g\in\G}\int_{0}^{\infty}|\langle f,e^{tA}g\rangle |^2dt\leq C\|f\|^2, \qquad \text {for all } f \in \ObsDynSp,
    \end{equation*} 
    for some $c, C>0$.
    To do that,  we estimate the difference $ \Delta$ in the following way:
         \begin{eqnarray*}
	    \Delta
	    &:=&\left|{\sum_{g\in\G}\int_{0}^{\infty}|\langle f,e^{tA}g\rangle|^2dt-\sum_{g\in\G}\sum_{j=1}^{\infty}\int_{t_j}^{t_{j+1}}\ |\langle f,e^{t_jA}g\rangle|^2}dt\right| \\
	    &\leq&\sum_{j=1}^{\infty}\int_{t_j}^{t_{j+1}}\sum_{g\in\G}\left| {|\langle f,e^{tA}g\rangle|^2-|\langle f,e^{t_jA}g\rangle|^2}  \right| dt\\
	    &\leq&\sum_{j=1}^{\infty}\int_{t_j}^{t_{j+1}}\sum_{g\in\G}\left(|\langle e^{tA^*}f,g\rangle|+|\langle e^{t_jA^*}f,g\rangle|\right)|\langle e^{tA^*}f-e^{t_jA^*}f,g\rangle|dt\\
	    &\leq&\sum_{j=1}^{\infty}\int_{t_j}^{t_{j+1}}\left(\sum_{g\in\G}(|\langle e^{tA^*}f,g\rangle|+|\langle e^{t_jA^*}f,g\rangle|)^2\right)^{1/2}\left(\sum_{g\in\G}|\langle e^{tA^*}f-e^{t_jA^*}f,g\rangle|^2\right)^{1/2}dt\\
	    &\leq&K\sum_{j=1}^{\infty}\int_{t_j}^{t_{j+1}}\left(\|e^{tA^*}f\|^2+\|e^{t_jA^*}f\|^2\right)^{1/2}\left(\|e^{tA^*}f-e^{t_jA^*}f\|^2\right)^{1/2}dt\\
	    &\leq &\sqrt{2}KM^2 \, \left(\sum_{j=1}^{\infty}e^{\omega t_j}\int_{t_j}^{t_{j+1}}\|e^{(t-t_j)\ObsDynOp}-I\|\ dt \right)\, \|f\|^2.
    \end{eqnarray*}
    Given $\varepsilon >0$ we take $\delta>0$  such that for $|t-t_{j}|<\delta$ we have 
    \begin{equation}\label{epsilon}
    \|e^{(t-t_j)\ObsDynOp}-I\|<\varepsilon.    
    \end{equation}
    Since $T$ is uniformly separated, $j \cdot \delta_0\leq t_j$
     for all $j\in\N$.
    Then, from the last inequality we get 
    \begin{align}\label{DiffContDiscDefinitive 2}
         \Delta &\leq\left(\sum_{j=1}^{\infty}e^{\omega t_j}\right) \left(\sqrt{2}KM^2\delta\right)\varepsilon\|f\|^2
         \leq \left( \frac{\sqrt{2}KM^2\delta}{1-e^{\omega\delta_0}}\right)\varepsilon\|f\|^2. \notag
    \end{align}
    Choosing  $\delta$ and $\varepsilon$ so small that  \eqref{epsilon} is satisfied and  $\left(\frac{\sqrt{2}KM^2\delta}{1-e^{\omega\delta_0}}\right)\varepsilon<c/2$, we achieve 
    $$\delta\sum_{g\in\G}\sum_{i=1}^{\infty}|\langle f,e^{t_iA}g\rangle|^2\geq c\|f\|^2-\frac{c}{2}\|f\|^2=\frac{c}{2}\|f\|^2.$$ Therefore,  for any uniformly separated countable set $T=\{t_j: j\in \N\}$ with $t_1=0$ and $0<t_{j+1}-t_j<\delta$ for all $ j\in \N$, the system  $\{e^{tA}g\}_{g\in\G,t\in T}$ is a frame for $\ObsDynSp$.\par
    
    $ii)\Rightarrow iii)$ Is trivial.
    
    $iii)\Rightarrow i)$ 
    Since $\{e^{tA}\}_{t\geq 0}$ is exponentially stable and $\G$ is Bessel with constant $K$, it holds that $\{e^{tA}g\}_{g\in\G,t\in[0,\infty]}$ is Bessel. Indeed, 
    \begin{equation*}
	\int_{0}^{\infty}\sum_{g\in\G}|\langle f,e^{tA}g\rangle|^2dt\leq K\int_0^\infty \|e^{tA^*}f\|^2dt\leq
	KM^2\left(\int_{0}^{\infty}e^{2\omega t}dt\right)\|f\|^2\leq \widetilde{K}\|f\|^2
	\end{equation*} 
	for a constant $0<\widetilde{K}<\infty$.
    
    Finally, let us see that $\{e^{tA}g\}_{g\in\G,t\in[0,\infty)}$ is bounded below under the assumption that there exists $m>0$ and a set of the form $T=\{t_j:= j\cdot m : j=0,1,\ldots \}$ such that  $\{e^{tA}g\}_{g\in\G,t\in T}$ is a frame for $\ObsDynSp$ with frame constants $c,C>0$  i.e., 
    \[c\|f\|^2\leq\sum_{g\in\G}\sum_{j=1}^{\infty}|\langle f,e^{t_jA}g\rangle|\leq C\|f\|^2 \qquad \text {for all } f \in \ObsDynSp. \]  
   We have that, 
    \begin{eqnarray*}
	    \sum_{g\in\G}\int_{0}^{\infty}|\langle f,e^{tA}g\rangle|^2dt&=&\sum_{g\in\G}\sum_{j=1}^{\infty}\int_{t_j}^{t_{j+1}}|\langle f,e^{tA}g\rangle|^2dt\\
	    &=&\sum_{g\in\G}\sum_{j=1}^{\infty}\int_{0}^{t_{j+1}-t_{j}}|\langle (e^{tA^*}f,e^{t_jA}g\rangle|^2dt\\
	  	    &\geq&\int_{0}^{m}c\|e^{tA^*}f\|^2dt\\
	    &\geq& \widetilde{c} \, \|f\|^2,
    \end{eqnarray*}
where $0< \widetilde{c}=c \cdot \frac{1-e^{-2m\|\ObsDynOp\|}}{2\|A\|}<\infty$. 
    This concludes the proof that  $\{e^{tA}g\}_{g\in\G,t\in[0,\infty)}$ is a semi-continuous frame for $\ObsDynSp$.     
\end{proof}

\begin{remarkx}\
\begin{enumerate}
\item As  in the finite time case, the statement $(ii)$ is equivalent to the same statement changing ``for any'' to ``for a''.
\item 
Unlike what happened in the case $L<\infty$, for the infinite time case, the Bessel condition of $\{e^{tA}g\}_{g\in \G, t\in[0,\infty)}$ is not trivially satisfied by requiring $\G$ to be a Bessel system. We mention that, by Datko's Theorem (cf. for eg. \cite[Corollary 1.1.14]{Tucsnak}),  the expression $\int_0^\infty \| e^{tA^*}f\|^2dt$,  is finite if and only if $\{e^{tA}\}$ is exponentially stable. Therefore, this proof does not remain valid if we remove the exponentially stable condition. 
\end{enumerate}
\end{remarkx}

\subsection{Discretization of the infinite time frame: Method 2.}
\

The following results, Theorem \ref{prop 6.5.2 tucsnak} and Proposition \ref{6.1.13 T}, are particular cases of  \cite[Propositions 6.5.2]{Tucsnak} and \cite[Proposition 6.1.13]{Tucsnak}. For completeness of the presentation we provide adaptations of the original proofs to our context since their original statements are written in the language of control theory instead of the language of dynamical sampling. For a full understanding on the relations between these two topics we refer the reader to \cite{dmm}.

\begin{theorem}\label{prop 6.5.2 tucsnak}
    Let $\ObsDynOp\in\mathcal{B}(\ObsDynSp)$ and let $\G$ a subset of vectors in $\HH$ such that 
	$\{e^{tA}g \}_{g\in\G,t\in[0,\infty)}$ is a semi-continuous frame for $\mathcal{H}$. If $\{e^{tA} \}_{t\in[0,\infty)}$ is exponentially stable, then there exist $0<L<\infty$ for which $\{e^{tA}g \}_{g\in\G,t\in[0,L]}$ is a semi-continuous frame for $\mathcal{H}$.
\end{theorem}
\begin{proof}
    One hypothesis says that there exist positive constants $c$ and $C$ such that
    \begin{equation*}
        c\|f\|^2\leq\int_{0}^{\infty}\sum_{g\in\G}|\langle f,e^{tA}g\rangle |^2 dt \leq C\|f\|^2.
    \end{equation*}
   Note that
    \begin{equation*}
        \int_{0}^{L}\sum_{g\in\G}|\langle f,e^{tA}g\rangle|^2dt
        \leq\int_{0}^{\infty}\sum_{g\in\G}|\langle f,e^{tA}g\rangle|^2dt \qquad \forall \, 0<L<\infty,
    \end{equation*}
    and therefore the Bessel condition follows immediately without using the assumption of having an exponentially stable semigroup. 
    
    To prove the lower bound,
    \begin{align*}
        \int_{0}^{L}\sum_{g\in\G}|\langle f,e^{tA}g\rangle|^2dt &=
            \int_{0}^{\infty}\sum_{g\in\G}|\langle f,e^{tA}g\rangle|^2dt -     \int_{L}^{\infty}\sum_{g\in\G}|\langle f,e^{tA}g\rangle|^2dt\\
            &\geq c\|f\|^2 - \int_{0}^{\infty}\sum_{g\in\G}|\langle f,e^{(t+L)A}g\rangle|^2dt\\
            &= c\|f\|^2 - \int_{0}^{\infty}\sum_{g\in\G}|\langle e^{LA^*}f,e^{tA}g\rangle|^2dt\\
             &\geq (c-C\|e^{LA}\|^2)\|f\|^2
    \end{align*}
    Since $e^{tA}$ is exponentially stable, there exists $\omega<0$ and $M\geq 1$ such that $\|e^{tA}\|\leq Me^{\omega t}$ for all $t\geq 0$. Hence taking $L$ sufficiently big such that   
     $(c-CMe^{2\omega L})>0$ we obtain
     \begin{equation*}
         \int_{0}^{L}\sum_{g\in\G}|\langle f,e^{tA}g\rangle|^2dt \geq (c-CMe^{2\omega L}) \|f\|^2 \qquad \forall f\in \HH.
     \end{equation*}
\end{proof}

\begin{proposition}\label{6.1.13 T} (Sufficient condition to be exponentially stable.)\
    Let $\ObsDynOp\in\mathcal{B}(\ObsDynSp)$ and let $\G$ be a Bessel system of vectors in $\ObsDynSp$ such that $\{e^{tA}g\}_{g\in\G, t\in [0,\infty) }$ is a Bessel system in $\ObsDynSp$ and  $\{e^{tA}g\}_{g\in\G, t\in [0,L] }$ is a semi-continuous frame  for $\ObsDynSp$ for some $0<L<\infty$. Then, $e^{tA}$ is an exponentially stable semigroup.    
\end{proposition}
\begin{proof}
    We will see that $\|e^{nLA}\|< 1$ for $n\in\N$ sufficiently large and then prove that this implies that the group is exponentially stable.
    
    Since $\{e^{tA}g\}_{g\in\G, t\in [0,L] }$ is a semi-continuous frame  for $\ObsDynSp$, there exist a constant $c>0$  such that for all $ f\in\ObsDynSp$ and for all $\tau\geq 0$
    \begin{equation*}
      c\|f\|^2\leq\int_0^L\sum_{g\in\G}|\langle f,e^{tA}g\rangle|dt=\int_\tau^{L+\tau}\sum_{g\in\G}|\langle f,e^{(t-\tau)A}g\rangle|dt.  
    \end{equation*}
    Since $\{e^{tA}g\}_{g\in\G, t\in [0,\infty) }$ is a Bessel system in $\ObsDynSp$, there exist a constant $C>0$  such that for all $ f\in\ObsDynSp$
   $
      \int_0^\infty\sum_{g\in\G}|\langle f,e^{tA}g\rangle|dt\leq C\|f\|^2.  
$
    Using  the previous equations we have
    \begin{align*}
        C\|f\|^2        &\geq\sum_{k\geq 0} \int_{kL}^{L+kL}\sum_{g\in\G}|\langle f,e^{tA}g\rangle|dt
        =\sum_{k\geq 0} \int_{kL}^{L+kL}\sum_{g\in\G}|\langle e^{kLA^*}f,e^{(t-kL)A}g\rangle|dt\\
        &\geq c\sum_{k\geq 0} \|e^{kLA^*}f\|^2.
    \end{align*}
    In particular, for all $n\in\N$ and for all $k\geq 0$
    \begin{equation*}\label{6.1.13 tucsnak 3}
        \sum_{k=1}^n\|e^{kLA^*}f\|^2\leq \frac{C}{c}\|f\|^2 \quad  \text{ and }  \quad 
         \|e^{kLA^*}f\|^2\leq \frac{C}{c}\|f\|^2 \quad (\forall f\in \mathcal{H}).
    \end{equation*}
     
      Therefore we obtain a bound for the operator norm
 $       \|e^{kLA}\|^2=\|e^{kLA^*}\|^2\leq\frac{C}{c} \ \forall k\geq 0$. 
    As before, using these equations we obtain for all $n\in\N$ and for all $f\in \mathcal{H}$
    \begin{equation*}
        \|e^{nLA^*}f\|^2=\frac{1}{n}\sum_{k=1}^n\|e^{(n-k)LA^*}e^{kLA^*}f\|^2=\leq \frac{C}{nc}\sum_{k=1}^n\|e^{kLA^*}f\|^2\leq \frac{C^2}{nc^2}\|f\|^2.
    \end{equation*}
    That is, for all $n\in \N$, $\|e^{nLA}\|=\|e^{nLA^*}\|\leq \frac{C}{\sqrt{n} \, c}$. Therefore there exist $n$ and $\omega <0$
  such that $\|e^{nLA}\|=\|e^{nLA^*}\|< e^{\omega nL}<1 $. Moreover, since $\|e^{tA^*}\|$ is continuous as a function of $t$ and $e^{\omega t}$ decreasing, we can find $M\geq 1$ satisfiying $\|e^{tA^*}\|< Me^{\omega t}$ for all $t\leq nL$. For $t>0$ let us write $t = m\cdot nL +t_0$, with $t_0<nL$. Therefore, 
    \begin{align*}
        \|e^{tA}\|=\|e^{tA^*}\|\leq \|e^{nLA^*}\|^m\|e^{t_0A^*}\|\leq M^2 e^{\omega(mnL + t_0)} = M^2 e^{\omega t}.
    \end{align*}
    This concludes the proof.
\end{proof}

As a result of Theorems \ref{ScToDscr} and  \ref{prop 6.5.2 tucsnak} we have the following corollary.

\begin{corollary}\label{exponentially stable if only if infinite implies finite}
     Let $\ObsDynOp\in\mathcal{B}(\ObsDynSp)$ and let $\G$ be a Bessel system of vectors in $\ObsDynSp$. If $\{e^{tA}g\}_{g\in\G, t\in [0,\infty) }$ is a semi-continuous frame in $\ObsDynSp$, then the following  are equivalent.
     \begin{enumerate}
         \item[i)]  There exists some  $0<L<\infty$ such that $\{e^{tA}g\}_{g\in\G, t\in [0,L] }$ is a semi-continuous frame  for $\ObsDynSp$.
         \item[ii)] There exists a finite partition $T=\{t_j:j=1,\ldots,n\}$ and $0= t_1< t_2<\ldots<t_n\leq L$ of $[0,L]$ such that $\{e^{tA}g \}_{g\in\G,t\in T}$ is a frame for $\ObsDynSp$.
        \item[iii)]  The semigroup $e^{tA}$ is exponentially stable. 
     \end{enumerate} 
\end{corollary}

Hence, as in Corollary \ref{frame tiempo finito implica G inf}, for exponentially stable semigroups we have the following corollary.

\begin{corollary}\label{exp stab implies G inf}
     If $A\in \mathcal{B}(\HH)$ is such that $\{e^{tA}\}_{t\in [0,\infty)}$ an exponentially stable semigroup and  $\{e^{tA}g\}_{g\in\G, t\in [0,\infty)}$ is a frame for an infinite dimensional Hilbert space $\HH$, then $|\G|=\infty$.
\end{corollary}

From this result we conclude, that if one expects to be able to sample using only a finite number of spatial points, then one needs to consider operators, that do not generate exponentially stable semigroups.

\section{Relations between continuous dynamical sampling and discrete dynamical sampling}\label{diagonalization}

\subsection{Bounded operators with an orthonormal or unconditional basis of eigenvectors}
\begin{theorem}\label{teo_bon} Let $\mathcal{H}$ be a separable Hilbert space and let $A\in\mathcal{B}(\mathcal{H})$, having an orthonormal basis of eigenvectors $\{e_j\}$ such that $Ae_j= -\lambda_j e_j$ with $\{\lambda_j\}\subset \C_+$. Let $g^i\in \mathcal{H}$ for $i\in I$ where $I$ is a countable set (finite or infinite) of indexes. Then, $\{e^{tA} g^i\}_{i\in I, \, t\in[0,\infty)}$ is a semi-continuous frame for $\mathcal{H}$ if and only if $\{(\mathfrak{h}(A))^n a^i\}_{i\in I, \, n\in \N\cup\{0\}}$ is a frame for $\mathcal{H}$, where $\mathfrak{h}$ is as in \eqref{M del D en C} and  $a^i$ are defined as 
    $a^i_j :=\frac{\sqrt{2}}{1+{\lambda_j}}  g_j^i$ (where $g^i_j:=\langle g^i,e_j\rangle$ and $a^i_j:=\langle a^i,e_j\rangle$).
\end{theorem}

For the proof we need the following two lemmas. 
The idea of the first one is taken from \cite{suarez}.

\begin{lemma}\label{teo_bon_lemma_disc}
Let $\mathcal{H}$ be a separable Hilbert space, $a\in\mathcal{H}$ and $A\in \mathcal{B}(\mathcal H)$ having an orthonormal basis  of eigenvectors. If $\{A^na\}_{n\in\N\cup\{0\}}$ is a Bessel system in $\mathcal{H}$, then 
    \begin{equation}\label{disc_suarez}
        \sum_{n=0}^\infty |\langle A^n a , c \rangle|^2=\left\|   \sum_{j} \overline{a_j} \,  c_j \, k^\D_{\eta_j}\right\|_{H^2(\D)}^2 \qquad \forall \, c\in\mathcal{H},
    \end{equation}
where $\{\eta_j\}_{j}$ is the set of eigenvalues of $A$ and  $a_j$ and $c_j$ denote the coordinates of $a$ and $c$ with respect to the orthonormal basis of eigenvectors of $A$.
\end{lemma}
\begin{proof}
 First we note that $\{\eta_j\}_j\subset \D$. Indeed, denote by $\{e_j\}$ the orthonormal basis of $\mathcal{H}$ of eigenvectors of $A$, $Ae_j=\eta_je_j$. Then, for each $j_0\in\N$,
 \begin{equation}\label{aux_geom}
     \sum_{n=0}^\infty |\langle A^n a , e_{j_0} \rangle|^2=\sum_{n=0}^\infty |\eta_{j_0}|^{2n}|a_{j_0}|^2.
 \end{equation}
 Since $\{A^na\}_{n\in\N\cup\{0\}}$ is Bessel, \eqref{aux_geom} is finite. Therefore,  $|\eta_{j_0}|<1$  for every $j_0\in\N$. 
 
 Now consider $c\in\mathcal{H}$. Then,
 \begin{align*}
 \sum_{n=0}^\infty |\langle A^n a , c \rangle|^2  &= \sum_{n=0}^\infty\sum_{k,j} \eta_j^n \overline{\eta_k}^n  \,  \overline{a_k} \,  c_k \, {a_j} \,  \overline{c_j}
= \sum_{k,j} \left(\sum_{n=0}^\infty\eta_j^n \overline{\eta_k}^n\right)  \,  \overline{a_k} \,  c_k \, {a_j} \,  \overline{c_j}\\
 &= \sum_{k,j}  \frac{1}{1-\eta_j \overline{\eta_k}} \,  \overline{a_k} \,  c_k \, {a_j} \,  \overline{c_j}= \sum_{k,j}  k^\D_{\eta_k}(\eta_j) \,  \overline{a_k} \,  c_k \, {a_j} \,  \overline{c_j}\\
 &= \sum_{k,j} \langle k^\D_{\eta_k},k^\D_{\eta_j} \rangle_{H^2(\D)} \overline{a_k} \,  c_k \, {a^i_j} \,  \overline{c_j}
 =\left\langle \sum_{k} \overline{a_k} \,  c_k \, k^\D_{\eta_k} , \ \sum_{j} \overline{a_j} \,  c_j \, k^\D_{\eta_j}\right\rangle_{H^2(\D)}\\
&= \left\|   \sum_{j} \overline{a_j} \,  c_j \, k^\D_{\eta_j}\right\|_{H^2(\D)}^2.
 \end{align*}
The change in the order of summation in the second equality is justified by first considering  finite sequences  $c^N = \{c_j\}_j$ such that $c_j = 0$ for all $ j \geq N$ for  $N\in\N$, and then using the assumption that $\{A^na\}_{n\in\N\cup\{0\}}$ is a Bessel system.
\end{proof}
\begin{lemma}\label{teo_bon_lemma_cont}
Let $\mathcal{H}$ be a separable Hilbert space, $f\in\mathcal{H}$ and $A\in \mathcal{B}(H)$ having an orthonormal basis of eigenvectors. If $\{e^{tA}g\}_{t\in[0,\infty)}$ is a Bessel system in $\mathcal{H}$, then 
\begin{equation}\label{continuo_suarez}
    \int_0^\infty |\langle e^{tA} g , c \rangle|^2 dt=2\pi\left\| \sum_{j}   \overline{g_j} \, c_j \, k^{\C_+}_{\lambda_j} \right\|_{H^2(\C_+)}^2 \qquad \forall \, c\in\mathcal{H},
\end{equation}    
where $\{-\lambda_j\}_{j}$ is the set of eigenvalues of $A$ and  $g_j$ and $c_j$ denote the coordinates of $g$ and $c$ with respect to the orthonormal basis of eigenvectors of $A$.
\end{lemma}
\begin{proof}
As in the previous lemma we first see that the Bessel condition together with the hypothesis of having an orthonormal basis $\{e_j\}$ of eigenvectors of $A$ easily imply that $\{\lambda_j\}\subset \C_{+}$. For each $j_0\in\N$,
 \begin{equation}\label{aux_geom_2}
     \int_0^\infty |\langle e^{tA} g , e_{j_0} \rangle| ^2 dt  =\left( \int_0^\infty e^{-2t\mathop{\mathrm{Re}}(\lambda_{j_0})}dt\right) |g_{j_0}|^2.
 \end{equation}
 By hypothesis, the left side of  \eqref{aux_geom_2} is finite. This holds if and only if
 $\mathop{\mathrm{Re}}(\lambda_{j_0})>0$, and this must be satisfied for every $j_0\in\N$. Therefore
\begin{align*}
\int_0^\infty |\langle e^{tA} g , c \rangle|^2 dt  &=\int_0^\infty\sum_{k,j} e^{-t\lambda_j}e^{-t\overline{\lambda_k}} g_j \, \overline{c_j} \, \overline{g_k} \,  c_k \, dt\\ 
&=\sum_{k,j}\left(\int_0^\infty e^{-t\lambda_j}e^{-t\overline{\lambda_k}}dt \right) g_j \, \overline{c_j} \, \overline{g_k} \,  c_k \\
&=\sum_{k,j}\frac{2\pi}{2\pi(\lambda_j+\overline{\lambda_k})} g_j \, \overline{c_j} \, \overline{g_k} \,  c_k
=2\pi\sum_{k,j}\langle k^{\C_+}_{\lambda_k}, k^{\C_+}_{\lambda_j}\rangle_{H^2(\C_+)} g_j \, \overline{c_j} \, \overline{g_k} \,  c_k\\
&=2\pi\left\| \sum_{j}   \overline{g_j} \, c_j \, k^{\C_+}_{\lambda_j} \right\|_{H^2(\C_+)}^2.
 \end{align*}
The change in the order of the sum and the integral in the second equality is justified in the same way as in the previous lemma.
\end{proof}

\begin{remarkx}\label{teo_bon_remark}
Let $\mathcal{H}$ be a separable Hilbert space, $a,g\in\mathcal{H}$ and $A\in \mathcal{B}(\mathcal H)$ having an orthonormal basis of eigenvectors. From the proofs of the previous lemmas we know that
\begin{itemize}
    \item $\{A^na\}_{n\in\N\cup\{0\}}$ is a Bessel system of $\mathcal{H}$ if and only if the eigenvalues of $A$ are in $\D$;
    \item $\{e^{-tA}g\}_{t\in[0,\infty)}$ is a Bessel system of $\mathcal{H}$ if and only if the eigenvalues of $A$ are in $\C_+$.
\end{itemize}
This remark is a particular case of what is proved  \cite[Section 3]{dmm} when $A$ has an unconditional basis (or a Riesz basis) of eigenvectors.
\end{remarkx}

\begin{proof}[Proof of Theorem \ref{teo_bon}]

First notice that since $A$ is bounded and $\{\lambda_j\} \subset \C_+$, there exists $M>0$ such that $1 < |1+\lambda_j| < M$ for all $j\in \mathbb{N}\cup\{0\}$. Therefore, $a^i\in \mathcal{H}$ if and only if $g^i\in \mathcal{H}$ for every $i\in I$. 

Consider an arbitrary vector $c\in\mathcal{H}$ with coordinates $c_j = \langle c,e_j\rangle$. Using \eqref{disc_suarez}, \eqref{continuo_suarez}  and the isometric isomorphism $V$ between $H^2(\D)$ and $H^2(\C_+)$,  specifically \eqref{k en k 2}, the conclusion follows since for each $i\in I$ we have
\begin{align*}
\int_0^\infty |\langle e^{tA} g^i , c \rangle|^2 dt  
 &=2\pi\left\| \sum_{j}   \overline{g_j^i} \, c_j \, k^{\C_+}_{\lambda_j} \right\|_{H^2(\C_+)}^2
 =\left\| \sum_{j} \sqrt{2\pi} \ \,  \overline{g_j^i} \, c_j \,  V^{-1}\left(k^{\C_+}_{\lambda_j}\right)\right\|_{H^2(\D)}^2\\
&= \left\| \sum_{j} \frac{\sqrt{2}}{1+\overline{\lambda_j}} \,  \overline{g_j^i} \, {c_j} \, k_{\mathfrak{h}(\lambda_j)}^{\D} \right\|_{H^2(\D)}^2
  =  
 \left\|   \sum_{j} \overline{a^i_j} \,  c_j \, k^\D_{\mathfrak{h}(\lambda_j)}\right\|_{H^2(\D)}^2\\
 &=\sum_{n=0}^\infty |\langle (\mathfrak{h}(A))^n a^i , c \rangle|^2.
\end{align*}
\end{proof}

In fact, Theorem~\ref{teo_bon} can be extended to the case when the operator 
$A$ has a Riesz basis (or an unconditional basis) of eigenvectors instead. That is, if we assume $A$ to be ``diagonalizable'' instead of being diagonal as in Theorem \ref{teo_bon}. 

\begin{theorem}\label{teo_bon1} Let $\mathcal{H}$ be a separable Hilbert space and let $A\in\mathcal{B}(\mathcal{H})$, having a Riesz basis of eigenvectors $\{e_j\}$ such that $Ae_j= -\lambda_j e_j$ with $\{\lambda_j\}\subset \C_+$. Let $f^i\in \mathcal{H}$ for $i\in I$ where $I$ is a countable set (finite or infinite) of indexes. Then, $\{e^{tA} g^i\}_{i\in I, \, t\in[0,\infty) }$ is a semi-continuous frame for $\ell^2$ if and only if $\{(\mathfrak{h}(A))^n a^i\}_{i\in I, \, n\in \N_0}$ is a frame for $\mathcal{H}$, where $\mathfrak{h}$ is as in \eqref{M del D en C} and $a^i$ are defined as 
    $a^i_j :=\frac{\sqrt{2}}{1+{\lambda_j}}  g_j^i$ (where $g^i_j:=\langle g^i,e_j\rangle$ and $a^i_j:=\langle a^i,e_j\rangle$). \end{theorem}

\begin{proof}
Let $\{e_j\}$ be the Riesz basis as in the theorem. Consider $\{e_j\}$ and $\{e_j^\prime\}$  biorthogonal systems. Writing $a=\sum_{j}a_je_j^\prime$, $g=\sum_{j}g_je_j^\prime$ and $c=\sum_{j}c_j^\prime e_j$, where $a_j = \langle a, e_j \rangle$, $a_j = \langle g, e_j \rangle$ and $c_j = \langle c, e^\prime_j \rangle$, identities (\ref{disc_suarez}) and (\ref{aux_geom_2}) transform into
\begin{gather*}\label{aux12}
        \sum_{n=0}^\infty |\langle A^n a , c \rangle|^2=\left\|   \sum_{j} \overline{a_j} \,  c_j^\prime \, k^\D_{\eta_j}\right\|_{H^2(\D)}^2 \qquad \forall \, c\in\mathcal{H},\\
        \int_0^\infty |\langle e^{tA} g , c \rangle|^2 dt=2\pi\left\| \sum_{j}   \overline{g_j} \, c_j^\prime \, k^{\C_+}_{\lambda_j} \right\|_{H^2(\C_+)}^2 \qquad \forall \, c\in\mathcal{H}.
    \end{gather*}
    Then, the proof concludes as in Theorem \ref{teo_bon}.
\end{proof}

\subsection{General normal operators}
\

All the previous results (Lemmas \ref{teo_bon_lemma_disc}, \ref{teo_bon_lemma_cont}, Remark \ref{teo_bon_remark} and Theorem \ref{teo_bon}) have been proven  only for those normal operators that are diagonal. In this section we show that in fact, by using the Spectral Theorem with multiplicity, they can be extended to any normal operator. Our main theorem is the following.


\begin{theorem} \label{Cont disc normal op}
    Let $\mathcal{H}$ be a complex separable Hilbert space and let $A\in \mathcal{B}(\mathcal{H})$. Let $g^i\in \mathcal{H}$ for $i\in I$ where $I$ is a countable set (finite or infinite) of indexes. Then, $\{e^{tA} g^i\}_{i\in I, \, t\in[0,\infty)}$ is a semi-continuous frame if and only if \newline
    $\{(\mathfrak{h}(A))^n a^i\}_{i\in I, \, n\in \N\cup\{0\}}$ is a frame, where $\mathfrak{h}$ is as in \eqref{M del D en C} and the $a^i$ are defined by $(U a^i) (z) := \frac{\sqrt{2}}{1+z}(Ug^i)(z)$ with $U$ as defined in the Spectral Theorem with multiplicity (Theorem \ref{spectral theorem}).
\end{theorem}

For the proof of this theorem we will need the following result that is an adaptation of \cite[Theorem 5.6]{ItNormOp}.

\begin{proposition}\label{cor53}
		Let $A$ be a bounded normal operator in an infinite dimensional Hilbert space $\HH$.  If  the system of vectors  $\{e^{tA}g \}_{g\in\G,t\in[0,\infty)}$  is a semi-continuous frame for $\mathcal{H}$ for some $ \G \subset \HH$ with $|\G| < \infty$, then  $A=\sum_j-\lambda_jP_j$ where $ P_j $ are orthogonal projections such that $rank\  P_j\leq |\G|$ (i.e. the global multiplicity of $A$ is less than or equal to $ |\G| $). 	
\end{proposition}
The proof of this proposition follows {\em almost} the same arguments than the ones used in \cite[Theorem 5.6]{ItNormOp}. We state the adaptations  of those results to our context as a reference for the reader.  Since the proofs are almost verbatim, we omit them here. An exception is Lemma \ref{Teo5.5 ItNormOp} which we proved using  a different argument. 

\begin{lemma}\label{mainthm}
 Let $\HH$ be a complex separable Hilbert space,  $A\in \mathcal B(\HH)$ be a normal operator,  and let $\G$ be a countable set of vectors in $\HH$ such that  $\left\{e^{tA}g\right\}_{ g \in \G,\;t\in[0,\infty)}$ is complete in $\HH$. Let  $\mu_\infty,\mu_1,\mu_2,\dots$ be the measures  given in Theorem \ref{spectral theorem} for the operator $A$. Then for every  $1\leq j\leq\infty$ and $\mu_j$-a.e. $z$, the system of vectors  $\{(Ug)_j(z)\}_{g\in \G}$ is complete in   $\ell^2(\Om_j)$.
\end{lemma}

\begin{lemma}\label{framecond}  Let $\HH$ be a complex separable Hilbert space,  $A\in \mathcal B(\HH)$ be a normal operator,  $\mu$ be  its scalar spectral measure, and $ \G $ a countable system of vectors in $ \HH $.
	If  $\{e^{tA}g \}_{g\in\G,t\in[0,\infty)}$ is complete in $\HH$  and for every $g\in \G$ the system  $\{e^{tA}g \}_{t\in[0,\infty)}$ is Bessel in $\HH$, then	$\mu\left(\C_+\right)=0$.
	Moreover, 
	If $\{e^{tA}g \}_{g\in\G,t\in[0,\infty)}$ is semi-continuous frame in $\HH$, then	$\mu\left( \overline{\C_+}\right)=0$.
		 
\end{lemma}

\begin{lemma}\label{Teo5.5 ItNormOp}  Let $\HH$ be a complex separable infinite dimensional Hilbert space,  $A\in \mathcal B(\HH)$ be a normal operator  and $ \G $ a \textbf{finite} subset of vectors in $ \HH $ such that  $\{e^{tA}g\}_{g\in\G, t\in[0,\infty)}$ is a semi-continuous frame for $\HH$. If is $\mu$ the scalar spectral measure of $A$ then for every $\varepsilon>0$, $$\mu(\{z\in\C: \, \mathop{\mathrm{Re}}(z)>-\varepsilon\})>0.$$


\end{lemma} 
\begin{proof} 
    We proceed by contradiction. Suppose that there exists $\varepsilon_0>0$ such that
$
        \mu\left(\{z\in\C: \, \mathop{\mathrm{Re}}(z)>-\varepsilon_0\}\right)=0.
$
  Since $A$ is a normal operator, this implies that
$
     \{z\in\C: \, \mathop{\mathrm{Re}}(z)>-\varepsilon_0\} \subset \rho(A),    
$
    where $\rho(A)$ is the resolvent set of $A$. 
Therefore, by Definition \ref{def_exp_estable}, there exists $\omega<0$ such that $\{e^{tA}\}_{t\in [0,\infty)}$ satisfies 
    \begin{equation*}
        \|e^{tA}\|\leq Me^{\omega t} \qquad \forall t\geq 0 \quad \text{ for some } M\geq 1. 
    \end{equation*}
   This means that $\{e^{tA}\}_{t\in[0,\infty)}$ is an exponentially stable semigroup. Since $\{e^{tA}g\}_{g\in\G, t\in[0,\infty)}$ is a semi-continuous frame for $\HH$, from Theorem~\ref{prop 6.5.2 tucsnak}, there exists a finite time $0<L<\infty$ such that $\{e^{tA}g\}_{g\in\G, t\in[0,L]}$ is a semi-continuous frame for $\HH$. 
   
   Now, by Theorem \ref{ScToDscr} there exists a {\bf finite} partition $T=\{t_j:j=1,\ldots,n\}$ and $0= t_1< t_2<\ldots<t_n\leq L$ of $[0,L]$ such that $\{e^{tA}g \}_{g\in\G,t\in T}$ is a frame for $\ObsDynSp$. This is a contradiction since $|\G|<\infty$, and $\HH$ is infinite dimensional.
\end{proof}

We are now ready to prove our theorem.

\begin{proof}[Proof of Theorem \ref{Cont disc normal op}]

    Since $A$ is a normal operator, from Proposition \ref{cor53} it follows that if $\{e^{tA} g^i\}_{i\in I, \, t\in[0,\infty)}$ is a semi-continuous frame for $\mathcal{H}$, then $A$ is a diagonal operator. In the same fashion,  from Theorem \cite[Theorem 5.6]{ItNormOp}  if $\{(\mathfrak{h}(A))^n a^i\}_{i\in I, \, n\in \N\cup\{0\}}$ is a frame for $\mathcal{H}$, then $\mathfrak{h}(A)$ (and as a consequence $A$) is a diagonal operator as well. Therefore,  this theorem is a consequence of Theorem \ref{teo_bon}. 
\end{proof}

The conclusion of Theorem \ref{Cont disc normal op} also remains valid without the hypothesis of $A$ being a normal operator if we require instead $A$ to have a Riesz basis (or an unconditional basis) of eigenvectors. That is, if we assume $A$ to be ``diagonalizable''  (Theorem \ref{teo_bon1}). 


\section{Necessary and sufficient conditions to solve the continuous dynamical problem in infinite time}\label{sec_suarez}

\begin{theorem}\label{suarez_cont}
Let $\mathcal{H}$ be a complex separable Hilbert space and let $A\in \mathcal{B}(\mathcal{H})$. Let $g^i\in \mathcal{H}$, $i=1, \dots, m$. Then $\{ e^{tA}g^i\}_{1\le i\le m, n\in \mathbb{N}\cup \{0\}}$ is a frame if and only if
\begin{enumerate}
\item[(i)] There exists an orthonormal basis $\{e_j\}_{j\in J}$ of $\mathcal H$ of eigenvectors of $A$ (where $J$ is a countable set) such that $Ae_j=-\lambda_j e_j$  and $A = \sum_{k} -\lambda_k P_k$ where $P_k$ are orthogonal projections such that rank$ (P_k) \leq m$ (i.e. the multiplicity of each eigenvalue is less or equal than $m$).
\item[(ii)]
\begin{itemize}
\item Each $g^i$ ($i=1, \ldots , m$) is given by 
\begin{equation}\label{niunii2}
    g_j^i:=\langle g^i,e_j\rangle =  d_j  \, \alpha^i_j \sqrt{ \mathop{\mathrm{Re}}(\lambda_j)} \qquad (\forall j\in J),
\end{equation}
where $C^{-1} \le d_j\le C\,$  for some $C\geq 1\,$ and $\,\sum_{i=1}^m|\alpha^i_j|^2 =1\,$
for all $j$.
\item The sequence of eigenvalues $\{ \lambda_j \}_{j\in J}$ of $-A$ lies on $\C_{+}$ and is such that $\sum_{j}  (1-|\mathfrak{h}(\lambda_j)|^2) \delta_{\mathfrak{h}(\lambda_j)}$ is a Carleson measure on $\D$ (where $\mathfrak{h}$ is as in \eqref{M del D en C}).
\item The  sequence $S:=\{ \eta_j:=\mathfrak{h}(\lambda_j) \}$  and the double sequence
$\{ \alpha^i_j : \, j\in J, \,1\le i\le m\}$  in \eqref{niunii2} satisfy the following two conditions:
\begin{enumerate}
\item[(1)] there is $\beta>0$ such that $\Delta(\eta_j , \beta)$ contains no more than m points of $S$
counting repetitions for all $j$.
\item[(2)] there is $0<\gamma <\beta$ such that
if\/ $\eta_{j_1} , \ldots, \eta_{j_p}$ ($p\leq m$) are the points of $S$ in
$\Delta(\eta_{j_1}, \gamma)$  counting repetitions, the related matrix satisfies $\forall (c_{j_1},\ldots, c_{j_p}) \in \mathbb{C}^p$
\begin{equation*}
D
\left\|
    \begin{bmatrix}
    c_{j_1} \\
    \vdots  \\
    c_{j_p} \\
  \end{bmatrix}
\right\|^2_{\mathbb{C}^p}
\le
\left\|
  \begin{bmatrix}
    \alpha^1_{j_1}  & ...  &\alpha^1_{j_p} \\
    \vdots     &            &  \vdots  \\
     \alpha^m_{j_1}  & ...  &\alpha^m_{j_p}\\
  \end{bmatrix}
  \begin{bmatrix}
    c_{j_1} \\
    \vdots \\
    c_{j_p} \\
  \end{bmatrix}
\right\|^2_{\C^{m \times 1}}
\end{equation*}
where $D>0$ does not depend on $p$ or the $\alpha$'s .
\end{enumerate}
\end{itemize}
\end{enumerate}
\end{theorem}

\begin{remarkx}
    In particular, when $m=1$, for $g^1_j=\sqrt{ \mathop{\mathrm{Re}}(\lambda_j)}$ the orbit $\{ e^{tA} g^1\}_{t\in[0,\infty)}$ is a semi-continuous frame if and only if the sequence of normalized reproducing kernels $\{{k^{\C_+}_{\lambda_j}}/{\|k^{\C_+}_{\lambda_j}\|} \}$ is a Riesz sequence. 
    This happens if and only if $\{ \lambda_j\}$ is an interpolating sequence. 
\end{remarkx}

\begin{proof}
    From Proposition \ref{cor53} (i) is a necessary condition to obtain a frame. Moreover, it allows us to assume without loss of generality that $A\in \mathcal{B}(\ell^2(J))$ is a diagonal operator with respect to the standard basis of $\ell^2(J)$ as in Theorem \ref{teo_imp_suarez}.

    Now, by Theorem \ref{teo_bon}, we know that  $\{e^{tA}g^i\}$ is a semi-continuous frame for $\mathcal{H}$ if and only if $\{(\mathfrak{h}(A))^na^i\}$ is frame for $\mathcal{H}$ with $g_j^i$ and  $a_j^i$  satisfying the relation $g_j^i=\frac{1+\lambda_j}{\sqrt{2}}a_j^i$.
    Also, from Theorem \ref{teo_imp_suarez}, $\{(\mathfrak{h}(A))^na^i\}$ is frame if and only if  $a_j^i$ satisfy (\ref{niunii}) with $\eta_j:=\mathfrak{h}(\lambda_j)$ and the conditions above. Notice that $\mathfrak{h}(\lambda_j) \subset \mathbb{D}$ if and only if $\lambda_j\in \C_+$ for every $j$.
    Therefore, we only need to prove (\ref{niunii2}). This follows from considering 
    \begin{align*}
        g_j^i&=\frac{1+\lambda_j}{\sqrt{2}}a_j^i = \frac{1+\lambda_j} d_j\alpha_j (1-|\mathfrak{h}(\lambda_j)|^2)^{\frac{1}{2}}\\
        &= d_j\alpha_j \sqrt{\frac{1+\lambda_j}{1+\overline{\lambda_j}}}\sqrt{2\mathop{\mathrm{Re}}(\lambda_j)},
    \end{align*}
    and noticing that 
    on one hand, since the only conditions on $d_j$ are to be bounded above and below it is equivalent to take either  $d_j$ or $\sqrt{2}d_j$. Further, on the other hand,  since $\left| \sqrt{ \frac{1+\lambda_j}{1+\overline{\lambda_j}}}\right| = 1$ it is equivalent to take $\alpha_j^i$ or $\sqrt{\frac{1+\lambda_j}{1+\overline{\lambda_j}}}\alpha_j^i$. 

This concludes the proof. 
\end{proof}

\bigskip

\textsc{Keywords:}
Dynamical Sampling, Continuous Frame, Control Theory, Interpolating Sequence, Hardy Space.
\textsc{MSC 2010 Classification}: {42C15, 93B05, 93B07}\\

 \bigskip

The research for this article has mainly been carried out while R. D\'i az Mart\'i n was  holding a postdoctoral fellowship from CONICET, and  I. Medri was holding postdoctoral fellowship from CONICET and Vanderbilt University. 
U. Molter was supported by Grants UBACyT 20020170100430BA (University of Buenos Aires), PIP11220150100355 (CONICET) and PICT 2014-1480.\\ 

\bigskip

\textsc{Roc\'io D\'iaz Mart\'in is at the Intituto Argentino de Matemática (IAM) -- CONICET, Saavedra 15, CABA, Argentina and Universidad Nacional de C\'ordoba, Ciudad Universitaria, C\'ordoba, Argentina.}

\textit{E-mail adress:} \url{rpd0109@famaf.unc.edu.ar}\\

\textsc{Ivan Medri is at Department of Mathematics, Vanderbilt University, Nashville, TN 37240-0001, USA.}

\textit{E-mail adress:} \url{cocarojasjorgeluis@gmail.com}\\

\textsc{Ursula Molter is at Dto. de Matem\'atica, FCEyN, Universidad de Buenos Aires and IMAS -- CONICET-UBA, Ciudad Universitaria, CABA, Argentina.}

\textit{E-mail adress:} \url{umolter@dm.uba.ar}

\end{document}